\documentclass[a4paper,11pt,reqno]{amsart}

\usepackage[T1]{fontenc}
\usepackage{amsmath,amssymb,amsthm,mathtools,hyperref,geometry,cleveref,mathrsfs,xcolor,microtype,tabularx,array,multirow,float,tikz}

\allowdisplaybreaks

\definecolor{DarkBlue}{rgb}{0.1,0.1,0.55}
\definecolor{DarkRed}{rgb}{0.55,0.1,0.1}

\hypersetup{colorlinks=true,linkcolor=DarkBlue,citecolor=DarkRed,urlcolor=DarkBlue}

\numberwithin{equation}{section}
\theoremstyle{plain}
\newtheorem{theorem}{Theorem}[section]
\newtheorem{proposition}[theorem]{Proposition}
\newtheorem{corollary}[theorem]{Corollary}
\newtheorem{lemma}[theorem]{Lemma}
\newtheorem{remark}[theorem]{Remark}

\newcommand{\R}{\mathbb R}
\newcommand{\Sph}{\mathbb S}
\newcommand{\Torus}{\mathbb T}
\newcommand{\lk}{\operatorname{Lk}}
\newcommand{\II}{\mathrm{II}}
\newcommand{\dd}{\,\mathrm d}
\numberwithin{equation}{section}

\newcommand{\eqnum}{\refstepcounter{equation}\tag{\theequation}}

\title[Closed asymptotic curves]{A solution to a problem of Ghomi and Raffaelli on closed asymptotic curves}
\author[W. R. Ding]{Weiran Ding}
\address{School of Mathematical Sciences, South China Normal University, Guangzhou 510000, P. R. CHINA.}
\email{dingwr0806@m.scnu.edu.cn}
\subjclass[2020]{53A04, 53A05; 57K10}
\keywords{Closed asymptotic curve, negative curvature, Gauss map, self-linking}

\usetikzlibrary{arrows.meta}
\providecommand{\eqnum}{\refstepcounter{equation}\tag{\theequation}}

\newcommand{\ACcornerdata}{%
-1/0/0,
-0.542771473/0/0,
-0.497048621/0.00000148/-0.000000749,
-0.451326298/0.000173101/-0.000082605,
-0.405621758/0.001375906/-0.000620845,
-0.36003758/0.004837122/-0.00206574,
-0.314838088/0.011645101/-0.004706168,
-0.270470286/0.022613478/-0.008640975,
-0.227536132/0.038266265/-0.013806097,
-0.186738782/0.058847258/-0.020007813,
-0.148817501/0.084337294/-0.026955098,
-0.114479903/0.114479903/-0.034291612,
-0.084337294/0.148817501/-0.041628125,
-0.058847258/0.186738782/-0.04857541,
-0.038266265/0.227536132/-0.054777127,
-0.022613478/0.270470286/-0.059942248,
-0.011645101/0.314838088/-0.063877055,
-0.004837122/0.36003758/-0.066517483,
-0.001375906/0.405621758/-0.067962378,
-0.000173101/0.451326298/-0.068500619,
-0.00000148/0.497048621/-0.068582475,
0/0.542771473/-0.068583224,
0/1/-0.068583224}

\newcommand{\ACvertexdata}{%
0/0/0/0/0/0/1/1/0/43.599945133/2/0,
1/43.599945133/0/2/0/1/0/0/-1/1/1/0.000006858,
2/43.599945133/-1/2/-1/0/-1/1/0/1/4/-1.979986283,
3/44.599945133/-1/6/-1/1/0/0/1/2.1/4/-1.009993142,
4/44.599945133/1.1/6/3/0/1/-1/0/1/1/11.54,
5/43.599945133/1.1/5/3/-1/0/0/-1/1/3.5/14.569993142,
6/43.599945133/0.1/5/-0.5/0/-1/-1/0/43.599945133/2/9.649986283,
7/0/0.1/3/-0.5/-1/0/0/1/1/2.5/-12.194979425,
8/0/1.1/3/2/0/1/-1/0/1/5/-9.229972567,
9/-1/1.1/-2/2/-1/0/0/-1/2.1/5/-7.219979425,
10/-1/-1/-2/-3/0/-1/1/0/1/2/-3.019986283,
11/0/-1/0/-3/1/0/0/1/1/3/-0.029993142}

\newcommand{\ACarea}{-0.068583224}
\newcommand{\ACmidX}{-0.114479903}
\newcommand{\ACmidY}{0.114479903}
\newcommand{\ACmidA}{-0.034291612}

\begin{document}

\begin{abstract}
We construct a smooth embedded open annulus in Euclidean three-space with strictly negative Gaussian curvature and injective Gauss map that contains a simple closed asymptotic curve of zero normal self-linking. This gives an affirmative answer to Problem~1.1 of Ghomi and Raffaelli. The construction starts with two simple planar curves whose tangent vectors are positive multiples of one another. Their mixed integral is made zero by rounding a pair of polygons with one free length parameter.
\end{abstract}

\maketitle

\section{Introduction}

Let $M\subset\R^3$ be an oriented surface with unit normal $N$. A regular curve $\gamma$ in $M$ is asymptotic if $\langle\gamma'',N\circ\gamma\rangle=0$. For a simple closed curve $\gamma$ and a unit normal field $n$ along it, write
\[
	\lk(\gamma,n)=\lk(\gamma,\gamma+\varepsilon n),
\]
where the right-hand side is the linking number of the two curves and $\varepsilon>0$ is sufficiently small. Ghomi and Raffaelli \cite[Problem 1.1]{GR} ask whether a closed asymptotic curve can have $\lk(\gamma,N\circ\gamma)=0$ when $M$ has strictly negative Gaussian curvature and its Gauss map is injective. We give an affirmative answer.

\begin{theorem}\label{thm:main}
	There is a $C^\infty$ embedded open annulus $M\subset\R^3$ with strictly negative Gaussian curvature and an injective Gauss map $N:M\to\Sph^2$, containing a smooth regular simple closed asymptotic curve $\gamma$ with
	\[
		\lk(\gamma,N\circ\gamma)=0.
	\]
\end{theorem}

Here an open annulus means a surface diffeomorphic to $\Sph^1\times(-1,1)$. In particular, $M$ has no boundary, and no completeness assertion is intended. The normal push-offs in the theorem are disjoint regular embeddings for every sufficiently small positive $\varepsilon$.

The classical flexibility problem for closed surfaces, listed as Problem 50 in Yau's problem section \cite{Yau}, asks whether a closed surface in $\R^3$ can admit a nontrivial continuous isometric deformation in the $C^2$ category; see also \cite[Section 1.1]{Ghomi}. A related rigidity problem concerns smooth tight surfaces. Nirenberg \cite{Nirenberg} sought to extend Alexandrov's rigidity theorem for analytic tight surfaces to the smooth category. His work led to the conjecture that a negatively curved annulus bounded by convex curves in planes tangent to the annulus along its boundary contains no closed asymptotic curve.

Ghomi and Raffaelli \cite[Notes 1.5 and 1.6]{GR} observed that, for an embedded annulus with these boundary conditions, the Gauss map is injective and every closed asymptotic curve has zero normal self-linking. A negative answer to their Problem 1.1 would therefore prove Nirenberg's conjecture in the embedded case. Theorem \ref{thm:main} gives an affirmative answer instead: without the prescribed boundary geometry, injectivity of the Gauss map and vanishing normal self-linking do not suffice to exclude closed asymptotic curves.

There is also a formulation in terms of space curves alone. Following Ghomi \cite[Section 1.2]{Ghomi}, a continuous binormal is a continuous unit field orthogonal to a continuous choice of planes containing $\gamma'$ and $\gamma''$. This definition allows the ordinary curvature to vanish. Problem~1.4 of \cite{Ghomi} asks whether an injective continuous binormal must have nonzero self-linking. Theorem \ref{thm:main} gives a counterexample even when the binormal is required to be a regular embedding.

\begin{corollary}\label{cor:binormal}
	There are smooth regular embeddings $\gamma:\Sph^1\to\R^3$ and $n:\Sph^1\to\Sph^2$ such that
	\[
		\langle n,\gamma'\rangle=\langle n,\gamma''\rangle=0,\qquad\lk(\gamma,n)=0.
	\]
\end{corollary}

Earlier constructions and restrictions on closed asymptotic curves are due to Kovaleva \cite{Kov68}, Panov \cite{Panov}, and Arnold \cite{Arnold}; see also the discussion in \cite{GR}. The examples compared in \cite[Section 4]{GR} exhibit the two properties in Theorem \ref{thm:main} separately: Kovaleva's example has zero linking but a noninjective normal curve, whereas Ghomi and Raffaelli's Example 4.2 has an embedded normal curve and linking number $2$. The point here is to obtain both properties at once.

Without regularity of the binormal, Ghomi's question already has a negative answer obtained from the first example of Ni, Zhang, and Zhang \cite[Section 2]{NZZ}; we give the short verification in Remark \ref{rem:continuous-binormal}. Their surface has zero Gaussian curvature along axes crossed by the closed curves. Their second example is a complete negatively curved immersion built from a curve with nonvanishing ordinary curvature and torsion, and therefore cannot have an injective normal along its asymptotic core \cite[Notes 5.1 and 5.2]{GR}. Neither construction gives Theorem \ref{thm:main}.

Aicardi \cite{Aicardi} constructed curves with nonvanishing curvature and torsion and arbitrary self-linking; see \cite[Section 4]{GR}. For the same reason, these curves cannot have an injective binormal. Our curve instead contains straight subarcs. In particular, it does not satisfy the hypothesis in Kovaleva \cite[p. 954]{Kov95} that the curvature have only finitely many zeros, each of finite order.

The construction reduces to finding two simple planar curves $P$ and $Q$ with $Q'=fP'$ for a positive function $f$ and with vanishing mixed integral $\int\det(Q,P')$. The zero integral closes a space-curve lift of $P$, while $Q$ prescribes its normal curve. A thin ruled strip then gives the annulus. The ruled construction is standard; see \cite[Note 2.1]{GR}. The main step is the choice of planar data, which we obtain by rounding two polygons simultaneously and adjusting a single length parameter. The annulus may also be chosen as a graph over a planar annulus with an injective gradient map; see Corollary \ref{add:cor:graph}.

We use $\det(a,b)=a_1b_2-a_2b_1$ for $a,b\in\R^2$, and $J(x,y)=(-y,x)$. In the construction, all parameters have period $24$, and $\Torus=\R/(24\mathbb Z)$. For a parametrized surface $F$ with unit normal $N$, our convention for the second fundamental form is $\II_{ij}=\langle F_{ij},N\rangle$.

\section{Construction of the planar curves}

The relation between the planar curves and the spatial construction explains the two conditions below. Suppose that $P,Q:\Torus\to\R^2$ are smooth regular closed curves with $Q'=fP'$ for a positive function $f$. We seek a lift $\gamma=(P,z)$ with unit normal
\[
	n=\frac{(-JQ,1)}{\sqrt{1+|Q|^2}}.
\]
The condition $\langle n,\gamma'\rangle=0$ forces
\[
	z'=\det(Q,P').\eqnum\label{add:eq:height}
\]
Moreover, $Q'=fP'$ gives $\langle n,\gamma''\rangle=0$. Thus the normal and asymptotic conditions are satisfied before any choice of the initial height. The remaining condition is that the lift close up:
\[
	z(24)-z(0)=I(P,Q),\qquad I(P,Q)\coloneqq\int_0^{24}\det(Q,P')dt.\eqnum\label{add:eq:period}
\]
In particular, $I(P,Q)$ measures the net change in height after one circuit. Simplicity of $P$ will ensure that the lift is embedded, while simplicity of $Q$ will ensure that its normal curve is embedded in the upper hemisphere.

The vanishing of this integral rules out some otherwise natural choices of planar curves. Following \cite{GR}, a planar curve is locally star-shaped if there is a point of the plane that lies on none of its tangent lines.

\begin{lemma}\label{add:lem:tangent-lines}
	Suppose that $P,Q:\Torus\to\R^2$ are smooth regular closed curves, $Q'=fP'$ with $f>0$, and $I(P,Q)=0$. The tangent lines of $P$ cover $\R^2$, as do the tangent lines of $Q$. In particular, neither curve is locally star-shaped, and neither can be the boundary of a convex domain.
\end{lemma}

\begin{proof}
	Integration by parts gives $I(P,Q)=\int_0^{24}\det(P,Q')\dd t$. Since both curves are closed, for every $a,b\in\R^2$ we have
	\[
		\begin{aligned}
			I(P,Q)&=\int_0^{24}f(t)\det(P(t)-a,P'(t))dt,\\
			I(P,Q)&=\int_0^{24}\frac{1}{f(t)}\det(Q(t)-b,Q'(t))dt.
		\end{aligned}
		\eqnum\label{add:eq:weighted-period}
	\]
	Indeed, the terms containing $a$ and $b$ vanish because
	\[
	\int_0^{24}fP'=\int_0^{24}Q'=0\quad\text{and}\quad\int_0^{24}f^{-1}Q'=\int_0^{24}P'=0.
	\]
	If $a$ lies on no tangent line of $P$, then $\det(P-a,P')$ is continuous and never zero. It therefore has a fixed sign, contradicting the first identity in \eqref{add:eq:weighted-period} and $f>0$. The second identity proves the same assertion for $Q$. Finally, an interior point of a convex domain lies on none of the tangent lines to its smooth boundary.
\end{proof}

For the projected core $P$, this agrees with the obstruction in \cite[Theorem 1.3]{GR}. Here the same integral argument applies to both planar curves. The nonconvex shapes used in the construction are therefore dictated by the period condition, rather than merely by the choice of polygonal coordinates.

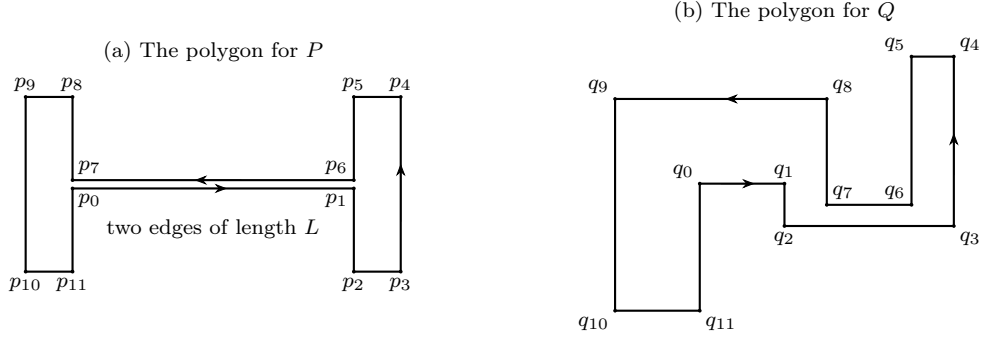
\begin{figure}[H]
\centering
\begin{minipage}[c]{0.49\linewidth}
\centering
\begin{tikzpicture}[x=0.62cm,y=1.1cm,font=\scriptsize,>=Stealth]
\draw[thick] (0,0)--(6,0)--(6,-1)--(7,-1)--(7,1.1)--(6,1.1)--(6,0.1)--(0,0.1)--(0,1.1)--(-1,1.1)--(-1,-1)--(0,-1)--cycle;
\draw[->] (2.3,0)--(3.3,0);
\draw[->] (3.6,0.1)--(2.6,0.1);
\draw[->] (7,-0.4)--(7,0.3);
\foreach \i/\x/\y/\pos in {0/0/0/below right,1/6/0/below left,2/6/-1/below,3/7/-1/below,4/7/1.1/above,5/6/1.1/above,6/6/0.1/above left,7/0/0.1/above right,8/0/1.1/above,9/-1/1.1/above,10/-1/-1/below,11/0/-1/below}{
  \fill (\x,\y) circle[radius=0.7pt];
  \node[\pos,inner sep=2pt] at (\x,\y) {$p_{\i}$};
}
\node at (3,-0.48) {two edges of length $L$};
\node at (3,1.65) {(a) The polygon for $P$};
\end{tikzpicture}
\end{minipage}\hfill
\begin{minipage}[c]{0.49\linewidth}
\centering
\begin{tikzpicture}[x=0.56cm,y=0.56cm,font=\scriptsize,>=Stealth]
\draw[thick] (0,0)--(2,0)--(2,-1)--(6,-1)--(6,3)--(5,3)--(5,-0.5)--(3,-0.5)--(3,2)--(-2,2)--(-2,-3)--(0,-3)--cycle;
\draw[->] (0.5,0)--(1.3,0);
\draw[->] (6,0.2)--(6,1.2);
\draw[->] (1.6,2)--(0.6,2);
\foreach \i/\x/\y/\pos in {0/0/0/above left,1/2/0/above,2/2/-1/below,3/6/-1/below right,4/6/3/above right,5/5/3/above left,6/5/-0.5/above left,7/3/-0.5/above right,8/3/2/above right,9/-2/2/above left,10/-2/-3/below left,11/0/-3/below right}{
  \fill (\x,\y) circle[radius=0.7pt];
  \node[\pos,inner sep=2pt] at (\x,\y) {$q_{\i}$};
}
\node at (2,4.05) {(b) The polygon for $Q$};
\end{tikzpicture}
\end{minipage}
\caption{The unrounded polygons. Edges with the same index are traversed in the same direction. The panel for $P$ is schematic: its long horizontal edges have been shortened, and its coordinate scales differ. The exact coordinates are those in Table \ref{tab:vertices}.}
\label{add:fig:polygons}
\end{figure}

\begin{proposition}\label{prop:planar}
	There are $C^\infty$ regular embeddings $P,Q:\Torus\to\R^2$ and a smooth function $f:\Torus\to(0,\infty)$ such that
	\[
		Q'=fP',\qquad\int_0^{24}\det(Q(t),P'(t))dt=0.\eqnum\label{eq:planar-data}
	\]
\end{proposition}

\begin{proof}
	For $L\in[43,44]$, consider the polygons with cyclic vertices $p_i=p_i(L)$ and $q_i$ in Table~\ref{tab:vertices}. The edge directions are $E=(1,0)$, $S=(0,-1)$, $N=(0,1)$, and $W=(-1,0)$. Indices are taken modulo $12$.
	
\begin{table}[H]
\centering
\renewcommand{\arraystretch}{1.15}
\begin{tabular}{c|c|c|c|c|c}
$i$ & $p_i$ & $q_i$ & $e_i$ & $a_i$ & $b_i$\\ \hline
0 & $(0,0)$ & $(0,0)$ & $E$ & $L$ & $2$\\
1 & $(L,0)$ & $(2,0)$ & $S$ & $1$ & $1$\\
2 & $(L,-1)$ & $(2,-1)$ & $E$ & $1$ & $4$\\
3 & $(L+1,-1)$ & $(6,-1)$ & $N$ & $21/10$ & $4$\\
4 & $(L+1,11/10)$ & $(6,3)$ & $W$ & $1$ & $1$\\
5 & $(L,11/10)$ & $(5,3)$ & $S$ & $1$ & $7/2$\\
6 & $(L,1/10)$ & $(5,-1/2)$ & $W$ & $L$ & $2$\\
7 & $(0,1/10)$ & $(3,-1/2)$ & $N$ & $1$ & $5/2$\\
8 & $(0,11/10)$ & $(3,2)$ & $W$ & $1$ & $5$\\
9 & $(-1,11/10)$ & $(-2,2)$ & $S$ & $21/10$ & $5$\\
10 & $(-1,-1)$ & $(-2,-3)$ & $E$ & $1$ & $2$\\
11 & $(0,-1)$ & $(0,-3)$ & $N$ & $1$ & $3$
\end{tabular}
\caption{The vertices and edge lengths of the two polygons: $p_{i+1}-p_i=a_i e_i$ and $q_{i+1}-q_i=b_i e_i$.}
\label{tab:vertices}
\end{table}

	Both polygons are simple, with separation bounds independent of $L$. For the first polygon, the two long horizontal edges lie at heights $0$ and $1/10$. The other edges lie in the rectangles
	\[
		[-1,0]\times[-1,11/10],\qquad[L,L+1]\times[-1,11/10].
	\]
	The vertical portions on each of the lines $x=0$ and $x=L$ are separated by a gap of $1/10$. It follows by comparing the edges that nonadjacent edges are at distance at least $1/10$, as is any vertex from a nonincident edge. For the second polygon, the horizontal and vertical edges are listed below by their supporting coordinates and intervals:
\[
\begin{array}{c|rrrrrr}
\text{height}&0&-1&3&-1/2&2&-3\\
\text{interval}&[0,2]&[2,6]&[5,6]&[3,5]&[-2,3]&[-2,0],\\[3pt]
\text{abscissa}&2&6&5&3&-2&0\\
\text{interval}&[-1,0]&[-1,3]&[-1/2,3]&[-1/2,2]&[-3,2]&[-3,0].
\end{array}
\]
	These intervals give the corresponding separation bound $1/2$. Adjacent edges of either polygon are orthogonal and meet only at their common vertex. We round corresponding corners in the same way. Choose a nonnegative smooth function $w$, supported in $[1/4,3/4]$ and symmetric about $1/2$, with $\int_0^1w=\pi/2$. Set
	\[
		\theta(u)=\int_0^u w(s)ds,\qquad h=\int_0^1\cos\theta(u)du,\qquad c=h^{-1}.
	\]
	Symmetry gives $\theta(1-u)=\pi/2-\theta(u)$, so $\int_0^1\sin\theta=h$. Since $\cos\theta+\sin\theta\geq1$, we have $c\leq2$. At the $i$-th corner use the arc
	\[
		r_i(u)=-e_{i-1}+c\int_0^u\left(\cos\theta(s)e_{i-1}+\sin\theta(s)e_i\right)ds,\qquad 0\leq u\leq1.\eqnum\label{eq:rounding}
	\]
	It satisfies
	\[
		r_i(0)=-e_{i-1},\qquad r_i(1)=e_i,\qquad |r_i'|=c.
	\]
	Near its endpoints, the velocity is respectively $ce_{i-1}$ and $ce_i$. In the orthonormal basis $(e_{i-1},e_i)$, the two coordinates of $r_i$ lie in $[-1,0]\times[0,1]$, and their sum is strictly increasing. Thus $r_i$ is simple and $|r_i|\leq\sqrt2$.\par
	Fix $\delta=1/100$. Replace the corners at $p_i$ and $q_i$ by $p_i+\delta r_i$ and $q_i+\delta r_i$, and retain the straight edges between their endpoints. The disks of radius $2\delta$ about the vertices of either polygon are disjoint and meet no nonincident edge. Each replacement arc lies in its vertex disk and, by the coordinate monotonicity just noted, meets its retained incident edges only at the endpoints. The rounded curves are therefore simple. Give each corner and the following straight segment a parameter interval of length $1$. On a corner use the parametrization above. On the $i$th straight segment choose a smooth nonnegative function $\psi$, compactly supported in $(0,1)$ with $\int_0^1\psi=1$, and prescribe the velocities
	\[
		\begin{aligned}
			P'&=\alpha_i(u)e_i,&
			\alpha_i(u)&=\delta c+(a_i-2\delta-\delta c)\psi(u),\\
			Q'&=\beta_i(u)e_i,&
			\beta_i(u)&=\delta c+(b_i-2\delta-\delta c)\psi(u).
		\end{aligned}
	\]
	Their integrals are $a_i-2\delta$ and $b_i-2\delta$, the lengths of the retained edges. Since $a_i,b_i\geq1$ and $c\leq2$, both speeds are positive. Near every join the velocities agree and are constant, so the resulting periodic parametrizations $P$ and $Q$ are smooth and regular. On corners $Q'=P'$, and on straight segments $Q'=(\beta_i/\alpha_i)P'$. Hence $Q'=fP'$ for a smooth positive function $f$, equal to $1$ near all joins. It remains to choose $L$ so that the mixed integral vanishes. For the polygons before rounding, its value is
	\[
		\begin{aligned}
			I_0(L)&=\sum_{i=0}^{11}a_i\det(q_i,e_i)\\
			&=0-2+1+\frac{63}{5}+3-5-\frac L2+3+2+\frac{21}{5}+3+0\\
			&=\frac{109}{5}-\frac L2.
		\end{aligned}
		\eqnum\label{eq:polygon-period}
	\]
	This is independent of the speeds along corresponding edges, since displacement along an edge of $Q$ is parallel to $dP$. For the rounded curves, the $i$th straight segment contributes $(a_i-2\delta)\det(q_i,e_i)$, while the $i$th corner contributes
	\[
		\delta\det(q_i,e_{i-1}+e_i)+\delta^2\int_0^1\det(r_i,r_i')du.
	\]
	The terms linear in $\delta$ cancel because
	\[
		\sum_i\det(q_i,e_{i-1})=\sum_i\det(q_{i-1},e_{i-1})=\sum_i\det(q_i,e_i).
	\]
	The first equality uses $q_i-q_{i-1}=b_{i-1}e_{i-1}$. Consequently,
	\[
		I_\delta(L)=\frac{109}{5}-\frac L2+\delta^2C,\qquad C=\sum_{i=0}^{11}\int_0^1\det(r_i,r_i')du.\eqnum\label{eq:smooth-period}
	\]
	The constant $C$ is independent of $L$, and
	\[
		|C|\leq12\sqrt2\,c\leq24\sqrt2<36.
	\]
	Choose
	\[
		L_*=\frac{218}{5}+2\delta^2C.\eqnum\label{eq:L-choice}
	\]
	Then $|L_*-218/5|<0.0072$, so $L_*\in[43,44]$ and all the preceding separation and regularity estimates apply. Equation \eqref{eq:smooth-period} gives $I_\delta(L_*)=0$, as required.
\end{proof}

The role of the long edges is particularly simple. In the unrounded pair, edge $0$ of $P$ has length $L$ but contributes zero to the mixed integral, since $q_0=(0,0)$. Edge $6$ also has length $L$, but its contribution is $-L/2$, since $q_6=(5,-1/2)$ and $e_6=(-1,0)$. All other contributions are independent of $L$. Thus, once the remaining edges have been chosen, varying $L$ changes the net height with slope $-1/2$. Figure~\ref{add:fig:polygons} shows the two polygons and their corresponding vertices. Rounding the two polygons independently would not in general preserve $Q'=fP'$. The common arc $r_i$ avoids this difficulty: on corresponding rounded corners the velocities are equal, and on the retained straight edges they are positive multiples of the same fixed vector. The endpoint velocities are constant on small intervals, so the pieces fit smoothly. Figure~\ref{add:fig:rounding} shows this replacement in the coordinates of the two incident edges.\par
There is also a useful distinction between the two effects of rounding on the period. The terms of order $\delta$ cancel after summing over the vertices, while the remaining term is $\delta^2C$, independent of $L$. Hence the adjustment from $218/5$ to $L_*$ in~\eqref{eq:L-choice} restores the period exactly. No additional connecting arc is needed in the spatial lift. For the drawings we use the following particular choice of the otherwise arbitrary function $w$:
\[
	\eta(u)=
	\begin{cases}
		\exp\!\left(-\dfrac{1}{1-16(u-1/2)^2}\right),& |u-1/2|<1/4,\\[4pt]
		0,& |u-1/2|\geq1/4,
	\end{cases}
	\qquad
	w(u)=\frac{\pi\eta(u)}{2\int_0^1\eta(s)ds}.\eqnum\label{add:eq:drawing-bump}
\]
This function has the support and symmetry required in the proof. Numerical integration gives $C\approx-0.2743328944$ and $L_*\approx43.5999451334$ for this choice. These approximations are used only to draw the curves; the construction and the closure condition use the exact value in \eqref{eq:L-choice}.

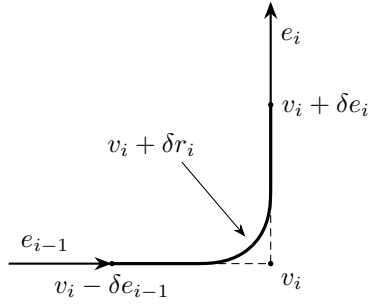
\begin{figure}[H]
\centering
\begin{tikzpicture}[x=2.1cm,y=2.1cm,font=\small,>=Stealth]
\draw[thick,->] (-1.65,0)--(-1,0);
\draw[thick,->] (0,1)--(0,1.65);
\draw[densely dashed] (-1,0)--(0,0)--(0,1);
\draw[very thick] (-1,0)
  foreach \ACx/\ACy/\ACa in \ACcornerdata {--(\ACx,\ACy)};
\fill (-1,0) circle[radius=1pt];
\fill (0,0) circle[radius=1pt];
\fill (0,1) circle[radius=1pt];
\node[below,inner sep=4pt] at (-1,0) {$v_i-\delta e_{i-1}$};
\node[right,inner sep=4pt] at (0,1) {$v_i+\delta e_i$};
\node[below right] at (0,0) {$v_i$};
\node at (-0.75,0.75) {$v_i+\delta r_i$};
\draw[->] (-0.56,0.64)--(-0.16,0.17);
\node[above] at (-1.42,0) {$e_{i-1}$};
\node[right] at (0,1.43) {$e_i$};
\end{tikzpicture}
\caption{The common corner replacement, with $v_i=p_i$ or $q_i$. The picture uses the local orthonormal coordinates $(e_{i-1},e_i)$ and scales $\delta$ to one. Dashed segments are removed; the rounded arc joins the retained incident edges with constant velocity near its endpoints. The displayed arc is sampled from~\eqref{add:eq:drawing-bump}.}
\label{add:fig:rounding}
\end{figure}

\section{Construction of the annulus}

We first record two elementary facts about thin neighborhoods of a compact curve. The uniformity in the first lemma will also be used to compute the linking number.

\begin{lemma}\label{lem:ribbon}
	Let $\gamma:\Sph^1\to\R^3$ be a smooth regular embedding, and let $V(t,r)$, $0\leq r\leq1$, be a smooth family of vector fields along $\gamma$. Suppose that $\gamma'(t)$ and $V(t,r)$ are linearly independent for every $(t,r)$. There is $a>0$, independent of $r$, such that
	\[
		F_r(t,u)=\gamma(t)+uV(t,r)
	\]
	is a smooth embedding of $\Sph^1\times[-a,a]$ for every $r$.
\end{lemma}

\begin{proof}
	The rank condition holds along $u=0$, hence on a uniform strip by compactness. If there were no uniform injective strip, we could find $r_j\to r_0$ and distinct pairs $(t_j,u_j),(s_j,v_j)$ with $u_j,v_j\to0$ and equal $F_{r_j}$-images. Passing to subsequences gives $t_j\to t_0$ and $s_j\to s_0$. The limiting equality $\gamma(t_0)=\gamma(s_0)$ implies $t_0=s_0$. Choose a linear projection $\pi:\R^3\to\R^2$ that is injective on the plane spanned by $\gamma'(t_0)$ and $V(t_0,r_0)$. In a local coordinate on the circle, the map
	\[
		(t,u,r)\longmapsto\left(\pi F_r(t,u),r\right)
	\]
	has invertible differential at $(t_0,0,r_0)$. The inverse function theorem contradicts the chosen collisions. At $r_0=0$ or $1$, apply it after a local smooth extension in $r$. Thus a uniform strip is injective. Since the closed strip is compact, each restriction is an embedding.
\end{proof}

\begin{lemma}\label{lem:normal-neighborhood}
	Let $Y$ be a smooth surface and $H:\Sph^1\times(-b,b)\to Y$ a smooth map. If $H(t,0)$ is injective in $t$ and $DH$ has rank two along $u=0$, then $H$ is an injective immersion on $\Sph^1\times[-a,a]$ for some $0<a<b$.
\end{lemma}

\begin{proof}
	Compactness gives the rank condition on a smaller strip. If no smaller strip were injective, collisions would converge to a pair of points of the core with the same image. Injectivity on the core makes their limits equal, contradicting the inverse function theorem in a surface chart at that point.
\end{proof}

\begin{proposition}\label{prop:annulus}
	Let $P,Q:\Torus\to\R^2$ be smooth regular embeddings satisfying~\eqref{eq:planar-data}. Define
	\[
		z(t)=\int_0^t\det(Q(s),P'(s))ds,\qquad\gamma(t)=(P(t),z(t)),\qquad n(t)=\frac{(-JQ(t),1)}{\sqrt{1+|Q(t)|^2}}.\eqnum\label{eq:lift}
	\]
	Set $v=|\gamma'|$, $T=\gamma'/v$, and $\nu=n\times T$. For some $a>0$, the ruled map
	\[
		F(t,u)=\gamma(t)+u\nu(t)\eqnum\label{eq:ruled-surface}
	\]
	embeds $\Torus\times[-a,a]$. Its open annulus has strictly negative Gaussian curvature and an injective Gauss map agreeing with $n$ along $\gamma$, and $\gamma$ is asymptotic.
\end{proposition}

\begin{proof}
	The zero integral in~\eqref{eq:planar-data} makes $z$ smooth and periodic. Since $P$ is a regular embedding, so is $\gamma$. The map
	\[
		\Phi(q)=\frac{(-Jq,1)}{\sqrt{1+|q|^2}}
	\]
	is a diffeomorphism from $\R^2$ to the upper open hemisphere, with inverse $\xi\mapsto J(\xi_1,\xi_2)/\xi_3$. Thus $n=\Phi\circ Q$ is a regular embedded spherical curve and $n_3>0$. The definition of $z$ gives
	\[
		\langle n,\gamma'\rangle=0,\qquad\langle n,\gamma''\rangle=\frac{\det(Q',P')}{\sqrt{1+|Q|^2}}=0.\eqnum\label{eq:orthogonality}
	\]
	It follows that $n'$ is orthogonal to both $n$ and $T$, so $n'=\lambda\nu$. In fact,
	\[
		\lambda=-\frac{fv}{1+|Q|^2}<0.\eqnum\label{eq:lambda}
	\]
	To check this, put $W=(-JQ,1)$ and $w=|W|$. Since $W\perp\gamma'$ and $Q'=fP'$, a scalar triple product gives
	\[
		\langle W',W\times\gamma'\rangle=-f|\gamma'|^2,\qquad\langle n',\nu\rangle=\frac{\langle W',W\times\gamma'\rangle}{w^2v}=-\frac{fv}{w^2}.
	\]
	No assumption on the ordinary curvature of $\gamma$ enters this computation. Along $u=0$ we have $F_t=vT$, $F_u=\nu$, and $F_t\times F_u=vn$. Lemma \ref{lem:ribbon} therefore gives an embedded closed strip. On a neighborhood of a sufficiently small such strip, its unit normal is
	\[
		\mathcal N(t,u)=\frac{F_t\times F_u}{|F_t\times F_u|},\qquad\mathcal N(t,0)=n(t).
	\]
	At the core, the first fundamental form is $\operatorname{diag}(v^2,1)$, and the second fundamental coefficients are
	\[
		\II_{tt}=\langle\gamma'',n\rangle=0,\qquad\II_{tu}=\langle\nu',n\rangle=-\lambda,\qquad\II_{uu}=0.
	\]
	Hence the core is asymptotic, and
	\[
		K(t,0)=-\frac{\lambda^2}{v^2}=-\frac{f(t)^2}{(1+|Q(t)|^2)^2}<0.\eqnum\label{eq:negative-curvature}
	\]
	By compactness, $K<0$ on a uniform smaller strip. To obtain injectivity of the Gauss map on the whole strip, write $\nu'=bT-\lambda n$. Differentiating the normalized cross product at the core gives
	\[
		\mathcal N_t(t,0)=\lambda\nu,\qquad\mathcal N_u(t,0)=\frac{\lambda}{v}T.\eqnum\label{eq:gauss-derivatives}
	\]
	Indeed, $\partial_u(F_t\times F_u)=\nu'\times\nu=bn+\lambda T$; normalization removes the component parallel to $n$ and divides the remaining component by $v$. The two derivatives in \eqref{eq:gauss-derivatives} are independent. Since $\mathcal N(t,0)=n(t)$ is injective, Lemma \ref{lem:normal-neighborhood} makes $\mathcal N$ injective on a smaller closed strip. Choose $a>0$ so that all these conclusions hold. Then $M=F(\Torus\times(-a,a))$ is a smooth embedded open annulus, and $N(F(t,u))=\mathcal N(t,u)$ defines the required injective Gauss map on $M$.
\end{proof}

\subsection{Curvature and geodesic torsion}

The planar curves control two different aspects of the lift: $P$ controls its position, while $Q$ controls its normal. Their curvatures are also related. Put
\[
	W=(-JQ,1),\qquad D=\det(P',P''),\qquad v=|\gamma'|.
\]
Since $Q'=fP'$, differentiating \eqref{eq:lift} gives $\gamma''=(P'',\det(Q,P''))$. A direct cross product then yields
\[
	\gamma'\times\gamma''=DW=|W|D\,n,\qquad\kappa_\gamma=\frac{|W||D|}{v^3},\eqnum\label{add:eq:core-curvature}
\]
where $\kappa_\gamma$ is the ordinary curvature of the spatial curve. Thus its curvature vanishes precisely where the curvature of $P$ vanishes. The two planar curves have their inflections at the same parameters, since
\[
	\det(Q',Q'')=f^2\det(P',P'').
\]
To compare this with the curvature of the surface, let $s$ denote arc length along $\gamma$. With $\nu=n\times T$, define $\kappa_g$ and $\tau_g$ by the Darboux equations
\[
	T_s=\kappa_g\nu,\qquad\nu_s=-\kappa_gT+\tau_g n,\qquad n_s=-\tau_g\nu.\eqnum\label{add:eq:darboux}
\]
These conventions agree with \cite[Section 2.2]{GR}. Equations \eqref{eq:lambda} and \eqref{add:eq:core-curvature} give
\[
	\kappa_g(s(t))=\frac{|W(t)|D(t)}{v(t)^3},\qquad\tau_g(s(t))=\frac{f(t)}{1+|Q(t)|^2}>0.\eqnum\label{add:eq:core-invariants}
\]
In particular, the vanishing of the ordinary curvature does not force the geodesic torsion to vanish. If $\sigma$ is arc length along $n$, then $\dd\sigma/\dd s=\tau_g$ and its unit tangent is $-\nu$. Its signed geodesic curvature in the sphere is consequently
\[
	\widetilde\kappa_g=\frac{\kappa_g}{\tau_g},
\]
as in \cite[Note 5.1]{GR}. On a retained straight interval, write $P'=\alpha e$ and $Q'=\beta e$, where $e$ is a fixed unit vector and $\alpha,\beta>0$. The quantity $\det(Q,e)$ is constant, so \eqref{eq:lift} becomes
\[
	\gamma'=\alpha\bigl(e,\det(Q,e)\bigr).
\]
The vector in parentheses is constant. Hence the corresponding part of $\gamma$ is a straight line segment, despite the possible variation of its speed. On the same interval $Q=q+\rho e$ for a fixed $q$, and therefore
\[
	(-JQ,1)=(-Jq,1)+\rho(-Je,0).
\]
The two vectors on the right are linearly independent and span a fixed plane through the origin. Thus $n$ traces a great-circle arc, and it does so regularly because $\tau_g>0$. A straight segment of the spatial core is therefore accompanied by a moving normal, rather than by a constant normal.\par
The straight pieces serve two purposes in the construction. They allow the speeds of $P$ and $Q$ to be adjusted independently without changing their tangent directions, and they join the rounded corners with all derivatives matching. The convex and concave corners give opposite signs of $\kappa_g$, separated by intervals on which it vanishes. This is consistent with the inflection restrictions in \cite[Notes 5.1 and 5.2]{GR}. Those restrictions do not require the curvature to vanish on an entire interval. The straight subarcs are a feature of this construction, not an additional necessary condition on every example. Figure \ref{add:fig:core-normal} displays the spatial core and its normal curve for the choice \eqref{add:eq:drawing-bump}. Matching labels mark the midpoints of the same rounded corners. The retained straight pieces of $\gamma$ and the corresponding great-circle pieces of $n$ are computed from the formulas above. The small rounded portions are shown by numerical sampling. The choice of $\psi$ changes only the traversal of the straight pieces and does not affect the images shown. These plots illustrate the construction; embeddedness and exact closure follow from the preceding proofs.
\begin{figure}[H]
\centering
\begin{minipage}[c]{0.49\linewidth}
\centering
\begin{tikzpicture}[x={(0.123101cm,-0.008130cm)},y={(0.021706cm,0.046128cm)},z={(0cm,0.115898cm)},font=\scriptsize,>=Stealth]
\foreach \ACi/\ACpx/\ACpy/\ACqx/\ACqy/\ACax/\ACay/\ACex/\ACey/\AClen/\ACblen/\ACzz in \ACvertexdata {
\pgfmathsetmacro{\ACturn}{\ACax*\ACey-\ACay*\ACex}
\pgfmathsetmacro{\ACzend}{\ACzz+0.01*(\ACqx*(\ACay+\ACey)-\ACqy*(\ACax+\ACex))+0.0001*\ACturn*\ACarea}

\draw[line width=0.55pt] ({\ACpx-0.01*\ACax},{\ACpy-0.01*\ACay},\ACzz)
foreach \ACrx/\ACry/\ACaa in \ACcornerdata {
--({\ACpx+0.01*(\ACrx*\ACax+\ACry*\ACex)},
{\ACpy+0.01*(\ACrx*\ACay+\ACry*\ACey)},
{\ACzz+0.01*(\ACqx*((\ACrx+1)*\ACay+\ACry*\ACey)-\ACqy*((\ACrx+1)*\ACax+\ACry*\ACex))+0.0001*\ACturn*\ACaa})
};

\draw[line width=0.55pt] ({\ACpx+0.01*\ACex},{\ACpy+0.01*\ACey},\ACzend)
--({\ACpx+(\AClen-0.01)*\ACex},{\ACpy+(\AClen-0.01)*\ACey},{\ACzend+(\AClen-0.02)*(\ACqx*\ACey-\ACqy*\ACex)});

\pgfmathtruncatemacro{\ACmark}{(\ACi==0)||(\ACi==3)||(\ACi==9)}
\ifnum\ACmark=1
\coordinate (ACpoint) at ({\ACpx+0.01*(\ACmidX*\ACax+\ACmidY*\ACex)},
{\ACpy+0.01*(\ACmidX*\ACay+\ACmidY*\ACey)},
{\ACzz+0.01*(\ACqx*((\ACmidX+1)*\ACay+\ACmidY*\ACey)-\ACqy*((\ACmidX+1)*\ACax+\ACmidY*\ACex))+0.0001*\ACturn*\ACmidA});
\fill (ACpoint) circle[radius=1pt];
\ifnum\ACi=0 \node[left] at (ACpoint) {$0$};\fi
\ifnum\ACi=3 \node[right] at (ACpoint) {$3$};\fi
\ifnum\ACi=9 \node[left] at (ACpoint) {$9$};\fi
\fi
}
\coordinate (ACorigin) at (17,0,-18);
\draw[->] (ACorigin)--++(7,0,0) node[right] {$x$};
\draw[->] (ACorigin)--++(0,7,0) node[left] {$y$};
\draw[->] (ACorigin)--++(0,0,7) node[above] {$z$};
\end{tikzpicture}

\small (a) The spatial core $\gamma$
\end{minipage}\hfill
\begin{minipage}[c]{0.49\linewidth}
\centering
\begin{tikzpicture}[scale=2,font=\scriptsize]
\draw[black!35] (0,0) circle[radius=1];
\draw[black!15,densely dashed] (0,0) circle[radius=0.5];
\draw[black!15,densely dashed] (0,0) circle[radius=0.8660254];
\draw[black!15,densely dashed] (-1,0)--(1,0) (0,-1)--(0,1);

\foreach \ACi/\ACpx/\ACpy/\ACqx/\ACqy/\ACax/\ACay/\ACex/\ACey/\AClen/\ACblen/\ACzz in \ACvertexdata {
\pgfmathsetmacro{\ACstartnorm}{sqrt(1+(\ACqx-0.01*\ACax)^2+(\ACqy-0.01*\ACay)^2)}

\draw[line width=0.55pt] ({(\ACqy-0.01*\ACay)/\ACstartnorm},{-(\ACqx-0.01*\ACax)/\ACstartnorm})
foreach \ACrx/\ACry/\ACaa in \ACcornerdata {
--({(\ACqy+0.01*(\ACrx*\ACay+\ACry*\ACey))/sqrt(1+(\ACqx+0.01*(\ACrx*\ACax+\ACry*\ACex))^2+(\ACqy+0.01*(\ACrx*\ACay+\ACry*\ACey))^2)},
{-(\ACqx+0.01*(\ACrx*\ACax+\ACry*\ACex))/sqrt(1+(\ACqx+0.01*(\ACrx*\ACax+\ACry*\ACex))^2+(\ACqy+0.01*(\ACrx*\ACay+\ACry*\ACey))^2)})
};

\draw[line width=0.55pt,domain=0:1,samples=31,variable=\ACt]
plot ({(\ACqy+(0.01+(\ACblen-0.02)*\ACt)*\ACey)/sqrt(1+(\ACqx+(0.01+(\ACblen-0.02)*\ACt)*\ACex)^2+(\ACqy+(0.01+(\ACblen-0.02)*\ACt)*\ACey)^2)},
{-(\ACqx+(0.01+(\ACblen-0.02)*\ACt)*\ACex)/sqrt(1+(\ACqx+(0.01+(\ACblen-0.02)*\ACt)*\ACex)^2+(\ACqy+(0.01+(\ACblen-0.02)*\ACt)*\ACey)^2)});

\pgfmathtruncatemacro{\ACmark}{(\ACi==0)||(\ACi==3)||(\ACi==9)}
\ifnum\ACmark=1
\pgfmathsetmacro{\ACmx}{\ACqx+0.01*(\ACmidX*\ACax+\ACmidY*\ACex)}
\pgfmathsetmacro{\ACmy}{\ACqy+0.01*(\ACmidX*\ACay+\ACmidY*\ACey)}
\pgfmathsetmacro{\ACmn}{sqrt(1+(\ACmx)^2+(\ACmy)^2)}
\coordinate (ACnormalpoint) at ({\ACmy/\ACmn},{-\ACmx/\ACmn});
\fill (ACnormalpoint) circle[radius=1pt];
\ifnum\ACi=0 \node[above right] at (ACnormalpoint) {$0$};\fi
\ifnum\ACi=3 \node[below] at (ACnormalpoint) {$3$};\fi
\ifnum\ACi=9 \node[above right] at (ACnormalpoint) {$9$};\fi
\fi
}
\end{tikzpicture}

\small (b) The spherical normal curve $n$
\end{minipage}
\caption{The spatial core and its normal curve for~\eqref{add:eq:drawing-bump}. The first panel is an orthogonal projection of $\gamma$; its planar projection is the rounded version of $P$ in Figure~\ref{add:fig:polygons}. The second panel shows the upper hemisphere viewed from $e_3$, with the equator as its outer circle. Labels $0$, $3$, and $9$ mark corresponding corner midpoints. The curves are numerically rendered, with the exact closing parameter specified by~\eqref{eq:L-choice}.}
\label{add:fig:core-normal}
\end{figure}

\subsection{Gaussian curvature}

The ruled parametrization also permits an explicit curvature calculation away from $\gamma$. In the arc-length parameter $s$, write
\[
	F(s,u)=\gamma(s)+u\nu(s),\qquad A(s,u)=1-u\kappa_g(s),\qquad H(s,u)=A(s,u)^2+u^2\tau_g(s)^2.
\]
The Darboux equations give
\[
	F_s=A T+u\tau_g n,\qquad F_u=\nu,\qquad F_s\times F_u=A n-u\tau_g T.
\]
Since $\tau_g>0$, the function $H$ is positive for every real $u$. Thus the extended ruled map is an immersion, with first fundamental form and unit normal
\[
	g_F=H\,\dd s^2+\dd u^2,\qquad\mathcal N(s,u)=\frac{A n-u\tau_g T}{\sqrt H}.\eqnum\label{add:eq:strip-metric}
\]
Here $F_{uu}=0$ and $F_{su}=-\kappa_g T+\tau_g n$, so
\[
	\II_{uu}=0,\qquad\II_{su}=\langle F_{su},\mathcal N\rangle=\frac{u\kappa_g\tau_g+A\tau_g}{\sqrt H}=\frac{\tau_g}{\sqrt H}.
\]
Consequently,
\[
	K(s,u)=-\frac{\tau_g(s)^2}{\bigl((1-u\kappa_g(s))^2+u^2\tau_g(s)^2\bigr)^2}<0.\eqnum\label{add:eq:strip-curvature}
\]
At $u=0$ this recovers \eqref{eq:negative-curvature}. On a straight subarc of the core it reduces to
\[
	K(s,u)=-\frac{\tau_g(s)^2}{\bigl(1+u^2\tau_g(s)^2\bigr)^2}<0.
\]
Thus a straight core segment lies on a strictly negatively curved part of the surface. The small width in Proposition \ref{prop:annulus} is needed for global embeddedness and injectivity of the Gauss map, not for the sign of the curvature of the ruled immersion.

\subsection{Representation as a graph}

The annulus can in fact be chosen to be a graph over a planar annulus. This gives another interpretation of the two planar curves and of the mixed integral.

\begin{corollary}\label{add:cor:graph}
	Let $P,Q,f$ be the data in Proposition \ref{prop:planar}, and let $z$ be defined by \eqref{eq:lift}. After restricting the annulus in Proposition \ref{prop:annulus}, there are a bounded planar annulus $\Omega\subset\R^2$ containing $P(\Torus)$ and a function $h\in C^\infty(\Omega)$ such that
	\[
		M=\{(x,h(x)):x\in\Omega\},\qquad\det D^2h<0\quad\text{on }\Omega,
	\]
	and $\nabla h:\Omega\to\R^2$ is injective. Along the core,
	\[
		\begin{aligned}
			h(P(t))&=z(t),&\nabla h(P(t))&=JQ(t),\\
			D^2h(P(t))P'(t)&=f(t)JP'(t),&\det D^2h(P(t))&=-f(t)^2.
		\end{aligned}
		\eqnum\label{add:eq:graph-data}
	\]
\end{corollary}

\begin{proof}
	Let $\pi:\R^3\to\R^2$ be the horizontal projection. On the core, $\pi F(t,0)=P(t)$, and
	\[
		\det\left((\pi F)_t,(\pi F)_u\right)(t,0)=(F_t\times F_u)_3(t,0)=v(t)n_3(t)>0.
	\]
	Since $P$ is an embedding, Lemma \ref{lem:normal-neighborhood}, applied with target $\R^2$, shows that $\pi\circ F$ is an injective immersion on a sufficiently small closed strip. Its restriction to the open strip is a diffeomorphism onto a bounded planar annulus $\Omega$. The third coordinate of $F\circ(\pi\circ F)^{-1}$ defines $h$. We may also shrink the strip so that its chosen normal has positive third component. The upward graph normal is then
	\[
		N(x,h(x))=\frac{(-\nabla h(x),1)}{\sqrt{1+|\nabla h(x)|^2}}.
	\]
	The map $p\mapsto(-p,1)/\sqrt{1+|p|^2}$ is a diffeomorphism from $\R^2$ onto the upper hemisphere. Hence injectivity of the Gauss map implies injectivity of $\nabla h$. The graph curvature formula
	\[
		K(x,h(x))=\frac{\det D^2h(x)}{(1+|\nabla h(x)|^2)^2}
	\]
	gives $\det D^2h<0$ on $\Omega$. Comparing the graph normal with \eqref{eq:lift} gives $\nabla h(P)=JQ$, and differentiating yields
	\[
		D^2h(P)P'=JQ'=fJP'.
	\]
	At each point of $P$, set $E=P'/|P'|$. In the orthonormal basis $(E,JE)$, symmetry of the Hessian therefore gives
	\[
		D^2h(P)=
		\begin{pmatrix}
			0&f\\
			f&\mu
		\end{pmatrix}
	\]
	for some real-valued function $\mu$. Its determinant is $-f^2$, proving \eqref{add:eq:graph-data}.
\end{proof}

Thus $P$ prescribes the position of the core in the base plane, and $JQ$ prescribes the gradient of the graph along it. The asymptotic condition is precisely $\langle D^2h(P)P',P'\rangle=0$. The period condition has the corresponding interpretation
\[
	\int_0^{24}\langle\nabla h(P(t)),P'(t)\rangle dt=\int_0^{24}\det(Q(t),P'(t))dt=0:
\]
the prescribed first derivatives must return the height to its initial value. The graph is defined on an annulus; no extension across the bounded complementary component is used.

\section{Proof of Theorem \ref{thm:main} and Corollary \ref{cor:binormal}}

We use the following elementary observation. Only the simplicity of the planar projection and the hemisphere condition on the normal are needed.

\begin{lemma}\label{lem:zero-linking}
	Suppose that $\gamma(t)=(P(t),z(t))$, where $P:\Torus\to\R^2$ is a smooth regular embedding and $z$ is smooth. Let $n:\Torus\to\Sph^2$ be a smooth field with $\langle n,\gamma'\rangle=0$ and $n_3>0$. Then $\gamma+\varepsilon n$ is a regular embedding disjoint from $\gamma$ for every sufficiently small $\varepsilon>0$, and $\lk(\gamma,n)=0$.
\end{lemma}

\begin{proof}
	Let $e_3=(0,0,1)$ and set
	\[
		V(t,r)=(1-r)n(t)+re_3,\qquad 0\leq r\leq1.
	\]
	Since
	\[
		\langle n,V(t,r)\rangle=(1-r)+rn_3(t)>0,
	\]
	each field $V(\cdot,r)$ is transverse to $\gamma'$. Lemma \ref{lem:ribbon} supplies $\varepsilon_0>0$ such that $\gamma$ and $\gamma+\varepsilon V(\cdot,r)$ are disjoint regular embeddings for all $r\in[0,1]$ and $0<\varepsilon<\varepsilon_0$. Homotopy invariance of linking therefore gives
	\[
		\lk(\gamma,\gamma+\varepsilon n)=\lk(\gamma,\gamma+\varepsilon e_3).
	\]
	For fixed $\varepsilon$ in this interval, flatten the latter pair by
	\[
		\alpha_r(t)=(P(t),(1-r)z(t)),\qquad\beta_r(t)=\alpha_r(t)+\varepsilon e_3,\qquad 0\leq r\leq1.
	\]
	Each component is a regular embedding, since its planar projection is $P$. The two components remain disjoint: equality of their planar coordinates forces equality of the parameters, but their heights then differ by $\varepsilon$. At $r=1$ they lie in two parallel planes and have linking number zero. Another application of homotopy invariance proves the lemma.
\end{proof}

\begin{proof}[Proof of Theorem \ref{thm:main} and Corollary \ref{cor:binormal}]
	Choose $P,Q,f$ from Proposition \ref{prop:planar}, and form the lift \eqref{eq:lift}. Proposition \ref{prop:annulus} gives the embedded negatively curved annulus and its injective Gauss map. Since $P$ is embedded and $n_3>0$, Lemma \ref{lem:zero-linking} gives zero normal self-linking. This proves the theorem. The same $\gamma$ and $n$ satisfy the corollary by \eqref{eq:orthogonality} and the regularity and injectivity of $n$.
\end{proof}

\begin{remark}\label{rem:continuous-binormal}
	To see why regularity of the binormal matters in Corollary \ref{cor:binormal}, consider
	\[
		\gamma_0(t)=(\cos t,\sin t,\cos2t),\qquad b_0(t)=\frac{(-4\cos^3t,4\sin^3t,1)}{\sqrt{1+16(\cos^6t+\sin^6t)}}.
	\]
	The numerator of $b_0$ is $\gamma_0'\times\gamma_0''$, so $b_0$ is a smooth unit binormal. It is injective: dividing its first two coordinates by its third recovers $-4\cos^3t$ and $4\sin^3t$, which determine $t$ modulo $2\pi$. Lemma \ref{lem:zero-linking} gives $\lk(\gamma_0,b_0)=0$. However, $b_0'(t)=0$ at the four points where $\sin2t=0$, so $b_0$ is not a regular embedding. This is the case $a=1$ of the closed asymptotic curves on the graph $z=x^4-y^4$ in \cite[Section 2]{NZZ}. It answers the continuous-binormal question as stated in \cite[Problem 1.4]{Ghomi}, but not its regular-binormal version.
\end{remark}

\raggedbottom
\par\bigskip
\noindent\textbf{Acknowledgements.}

This paper arose in close collaboration with ChatGPT 6-Astra. The collaboration involved exploring proof strategies, refining arguments, and revising the manuscript. The author takes full responsibility for the mathematical content and the final text.

\end{document}